\documentclass{article}

\usepackage{arxiv}

\usepackage[utf8]{inputenc} 
\usepackage[T1]{fontenc}    
\usepackage{hyperref}       
\usepackage{url}            
\usepackage{booktabs}       
\usepackage{amsfonts}       
\usepackage{nicefrac}       
\usepackage{microtype}      
\usepackage{lipsum}
\usepackage{amsmath,amsthm,amssymb}
\usepackage[all]{xy}
\usepackage{graphicx}
\usepackage{mathrsfs}
\usepackage{ upgreek}
\theoremstyle{plain}
\theoremstyle{remark}

\newtheorem*{convention*}{Convention}
\theoremstyle{plain}
\newtheorem{theorem}{Theorem}[section]

\newtheorem{Note}[theorem]{Note}

\theoremstyle{definition}
\newtheorem{definition}[theorem]{Definition}
\newtheorem{example}[theorem]{Example}

\title{More on fuzzy Topological Spaces on fuzzy space}

\author{Abd Ulazeez Alkouri\\
  Department of Mathematics\\
 Ajloun National University\\
P.O.Box: 43- Ajloun- 26810- Jordan.\\
  \texttt{1alkouriabdulazeez@gmail.com} \\
   \And
 Mohammad Hazaimeh\\
  Department of Mathematics\\
 Al-Hussein Bin Talal University\\
Ma'an, P.O. Box (20), 71111, Jordan\\
  \texttt{mhazaimeh06@gmail.com}  \\
  \And
  Ibrahim Jawarneh\\
  Department of Mathematics\\
 Al-Hussein Bin Talal University\\
Ma'an, P.O. Box (20), 71111, Jordan\\
  \texttt{ibrahim.a.jawarneh@ahu.edu.jo} \\
}

\begin{document}
\maketitle

\begin{abstract}
The fuzzy topological space was introduced by Dip in 1999 depending on the notion of fuzzy spaces. Dip’s approach helps to rectify the deviation in some definitions of fuzzy subsets in fuzzy topological spaces. In this paper, further definitions, and theorems on fuzzy topological space fill the lack in Dip’s article. Different types of fuzzy topological space on fuzzy space are presented such as co-finite, co-countable, right and left ray, and usual fuzzy topology. Furthermore, boundary, exterior, and isolated points of fuzzy sets are investigated and illustrated based on fuzzy spaces. Finally, separation axioms are studied on fuzzy spaces. 
\end{abstract}


\section{Introduction} 
Fuzzy set was introduced in 1965 by Zadeh \cite{Zadeh_1965}. Then, Days rolled over and the concept of fuzzy mathematics has been generalized and created several applications in a comprehensive range of areas. The fuzzy mathematics carries out the liability of translation from the human being's information/ knowledge to the fuzzy case. The competition lies in how to choose the logical development from several approaches. In 1994, Dip \cite{Dib_1994} selected the development of fuzzy sets in the field of algebra to illustrate his idea and step out to apply it to the field of topology in 1999 \cite{Dib_1999}. His idea centered on replacing fuzzy spaces instead of ordinary universal sets to correct the deviation in the definitions of the closed fuzzy subset and the closure of a fuzzy subset in fuzzy topological spaces.

Chang's approach in 1968 \cite{Chang_1968}, overcame the shortage of the fuzzy universal set by introducing fuzzy topology on a base set $X$, as a family $\tau$ of fuzzy subsets of $X$, which satisfies formal topology's conditions, as a generalization of Zadeh's concept \cite{Zadeh_1965}. Later on, some generalization, redefine, and suggestion on fuzzy topology was made and studied by Wong (1973) \cite{Wong_1973}, Lowen (1976) \cite{Lowen_1976}, and Hazra et al. (1992) \cite{Haz_Sam_1992}. In contrast to Chang's approach, Dib \cite{Dib_1999} introduced fuzzy topology on fuzzy space $(X,I)$ which is a collection of fuzzy subspaces, adequate the standard axioms of 
topology

In this research, we follow Dib's approach by using fuzzy space instead of universal set. Also, we present further definitions, examples, and theorems on fuzzy topological space and its properties related to point-set concepts. Then, separation axioms are introduced and studied on fuzzy spaces. Finally, the conclusion is summarized.
\section{Preliminaries}
A fuzzy set is a class of objects with a continuum of grades of membership function $u_A(x)$ which associates with each point in $X $a real number in the interval $[0,1]$, see \cite{Zadeh_1965}.In topology Chang \cite{Chang_1968}, presented the fuzzy topology (C-fuzzy topology) on a base set $X$ as in the following definition
\begin{definition}
	A fuzzy topology is a family $\sigma$ of fuzzy sets in $X$ which satisfies the following conditions:
	\begin{enumerate}
		\item $\emptyset, X \in \sigma$, 
		\item If $A,B \in \sigma$ , then $A \cap B \in \sigma$, 
		\item If $A_i \in \sigma$ for each $i \in I$, then $U_I A_i\ \in \sigma$.
	\end{enumerate}
$\sigma$ is called a fuzzy topology for $X$, and the pair $(X, \sigma)$ is a fuzzy topological space (FTS).
\end{definition}
\begin{definition}
\cite{Dubois1980}Triangular Fuzzy Number $A$ fuzzy number $\bar{A} = (a, b, c)$ is called a triangular fuzzy number if its membership function is given by.
\[	\mu_{\bar{A}}(x)= 
\begin{cases}
0, \; \;\; x \le a \\
\frac{x-a}{b-a}, \;\;\; a<x<b \\
\frac{c-x}{c-b}, \;\;\; b<x<c \\
0, \;\;\; x \ge a
\end{cases}
\]
A new approach to define the fuzzy topology using the concept of fuzzy space. Let $X$ be an ordinary set and $I$ be a closed interval $[0,1]$. The following definitions and examples are in Dib 1999 \cite{Dib_1999}.
\end{definition}
\begin{definition}
	The fuzzy space $(X,I)$ is the set of all ordered pairs $(x,I)$; $x \in X$, i.e. $\{(x,I): x \in X\}$, where $(x,I) = \{(x, r): r \in I\}$.The order pair $(x,I)$ is called a fuzzy element of the fuzzy space $(X,I)$.
\end{definition}
\begin{definition}
	The fuzzy subspace $U$ of the fuzzy space $(X,I)$ is a collection of ordered pairs $(x, u_x)$, where $x \in U_0$ (for a given subset $U_0$ of $X$) and $u_x$ is a subset of $I$, which contains at least one element besides zero element. If $x \not \in  U_0$, then $u_x = \{0\}$.
	The fuzzy subspace $U$ is denoted by $U = \{(x, u_x): x \in U_0\}$, where $(x, u_x)$ is the fuzzy element of fuzzy subspace $U$. $U_0$ is called the support of $U$ and denoted by $SU = U_0$.
	
	The empty fuzzy subspace $\emptyset$ of $(X,I)$ is defined by $\{(x,\phi_x): x \in \emptyset\}$, i.e. $S \emptyset = \emptyset$.
\end{definition}
\begin{definition}
	The fuzzy subspace $V = \{(x, v_x): x \in V_0\}$ is contained in fuzzy subspace $U = \{(x, u_x): x \in U_0
	\}$ and denoted by $V \subset U$, if $V_0 \subset U_0$ and $v_x \subset u_x$, for all $x \in V_0$	.
\end{definition}
\begin{definition}
	The fuzzy point $P$ of the fuzzy space $(X,I)$ is a fuzzy subspace $P = \{(x, p_x): x \in P_0
	\}$, where $P_0$ is a non-empty subset of $X$ and $p_x$ contains only one element $\rho_x$ besides zero: $p_x = \{0; \rho_x\}; x \in P_0$. The fuzzy point $P$ is contained in the fuzzy subspace $U$ and we write $P \in U$, if $p_x \subset u_x$, for all $x \in P_0$.
\end{definition}
\begin{definition}
	Let $U = \{(x, ux,): x \in U_0\}$ and $V = \{(x, vx): x \in V_0\}$ be fuzzy subspaces of the fuzzy spaces $(X,I)$. The union $U \cup V$ and the intersection $U \cap V$ of fuzzy subspaces are defined by the relation

	$U \cup V = \{(x,u_x \cup v_x):x \in U_0 \cup V_0\},$ 
	
		$U \cap V = \{(x,u_x \cap v_x):x \in U_0 \cap V_0\}.$
		
		The support of these fuzzy subspaces satisfies that
		
		$S (U \cup V) =S(U) \cup S(V) = U_0 \cup V_0$ 
		
				$S (U \cup V)  \subset S(U) \cup S(V) = U_0 \cup V_0$ 
				
the inclusion relation will be an equality, if $u_x \cap v_x \not = \{0\},$ for all $x \in U_0 \cap V_0$ .
\end{definition}
\begin{definition}
	The family $\tau$ of fuzzy subspaces of the fuzzy space $(X,I)$ is called a fuzzy topology on the fuzzy space  $(X,I)$, if $\tau$ satisfies the following conditions:

	\begin{enumerate}
		\item $(X,I) \in \tau$ and $\emptyset	\in \tau$,
		\item $U \cap V \in \tau$ for every $U,V \in \tau$,
		\item $U_{U \in \tau_1} U \in \tau$ for every $\tau_1 \subset \tau$.
	\end{enumerate}
The ordered pair $((X,I),\tau)$ is called a fuzzy topological space. The elements of $\tau$ are open fuzzy subspaces of the fuzzy topology.
\end{definition}
\begin{example}
	The trivial fuzzy topology $\tau$ on the fuzzy space $((X,I),\tau)$ contains only two elements $(X,I)$ and $\phi$.
\end{example}
\begin{example}
The discrete fuzzy topology $\tau$ on the fuzzy space $((X,I),\tau)$ contains all the fuzzy subspaces of $(X,I)$.
\end{example}
\begin{definition}
The neighborhood of the fuzzy point $P$ (or the fuzzy subset $A$) in the fuzzy topology $\tau$ is a fuzzy subspace $U$, which contains an element of $\tau$, containing $P$ (or $A$, respectively).
If $U$ is a neighborhood of the fuzzy point $P$ (or fuzzy subset $A$), then $P$ (or $A$) is called an interior point (or fuzzy subset) of $U$. The interior $U^0$ of the fuzzy subspace $\cup$ is the union of all its interior points. 
\end{definition}
\begin{Note}
	Remember that the operations "$\subset, \cup , \cap -$" on the fuzzy subspaces are defined through the corresponding operations "$\subset, \cup , \cap -$" on the sets of membership values. This helps us to reformulate the concepts and the results of the ordinary case and carry it to the fuzzy topological spaces.
\end{Note}
\begin{theorem}
	If $((X,I),\tau)$ is a fuzzy topological space, then the fuzzy subspace $U$ is open iff it is a neighborhood of all its fuzzy points
\end{theorem}
\begin{definition}
	The fuzzy subspace $U$ is called a closed fuzzy subspace in fuzzy topological space  $((X,I),\tau)$ if its complement $U^c = (X,I)- U$ is an open fuzzy subspace.

\end{definition}
\begin{definition}
	The closure $(\tau$ - closure) $U$ of the fuzzy subspace $U$ of the fuzzy topological space , $((X,I),\tau)$ is the intersection of all closed fuzzy subspaces containing $U$.
\end{definition}
\begin{definition}
	The fuzzy point $P$ is called a limit point of a fuzzy subspace $U$, if every neighborhood of $P$ contains fuzzy points of $U$ other than $P$. The fuzzy subset $A$ is called a limit fuzzy subset of $U$ if its associated fuzzy point $P_A$ is a limit point of $U$.
\end{definition}
\section{Main Results}
In this section, several definitions, examples, and theorems are derived based on fuzzy spaces (under Dib's 
approach). We define more topologies on fuzzy space. In the ordinary case the difference between two ordinary subsets is an ordinary subset; similarly, in the fuzzy case the difference between fuzzy subspaces will be a fuzzy subspace.
\begin{definition}
	The difference $U-V$ between the fuzzy subspaces $U$ and $V$ is defined by
	$U-V =\{(x,h_x):x \in U_0-V_0\},$ where $h_x= (u_x -v_x) \cup \{0\}$ Notice that $S(U-V) \supset U_0-V_0$ and the equality holds if $u_x \subset v_x$, for all $x \in U_0 \cap V_0$
\end{definition}
\begin{definition}
	The fuzzy subspace $U$ of the fuzzy space $(X,I)$ is finite if $U_0$ is finite.
\end{definition}
\begin{definition}
The fuzzy subspace $U$ of the fuzzy space $(X,I)$ is countable if $U_0$ is countable.
\end{definition}
\begin{definition}
	Let $(X,I)$ be a non-empty fuzzy set and $\tau = \{(X,I), \phi,U \subset (X,I); U^c \;\; \text{fuzzy finite}\}$, then $\tau$ is called the co-finite fuzzy topology on $(X,I)$, where $U^c = (X,I)-U$, is a finite set.
\end{definition}
\begin{example}
	Any finite set $X$, then the fuzzy topological space $\tau$ of the fuzzy space $(X,I)$ is co-finite.
\end{example}
\begin{definition}
	Let $(X,I)$ be a non-empty fuzzy set and $\tau = \{(X,I), \phi,U \subset (X,I); U^c$  $\text{fuzzy countable}\}$, then $\tau$ is called the co-countable fuzzy topology on $(X,I)$ , where $U^c = (X,I)- U$, is a countable set.
\end{definition}
\begin{example}
	Any countable set $X$, then the fuzzy topological space $\tau$ of the fuzzy space $(X,I)$ is co-countable.
\end{example}
\begin{definition}
	Let $(R,I)$ and $\tau_r = \{(R,I), \phi,U_r = ((r, \infty), u_x): u_x \subset I, x \in (r, \infty), r \in R\}$, then $\tau$ is called the right ray fuzzy topology on $(R,I)$.
\end{definition}
\begin{example}
	Consider $U_{r1} , U_{r2} , U_{r3}$ are in a right ray fuzzy topology on $(R,I)$ such that $S(U_{r1}) = (a, \infty), S(U_{r2}) =(b, \infty), S(U_{r3}) = (b, \infty),$ where $a < b < c$ with fuzzy membership values are given in the Graph 1.
\end{example}
\includegraphics[height=5cm,width=5 in]{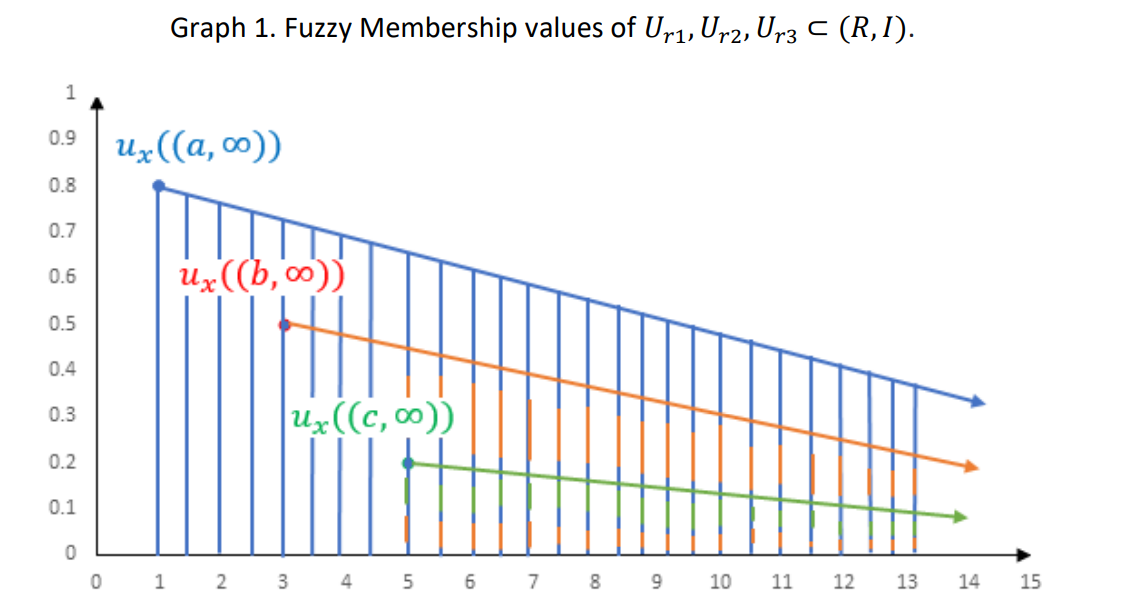}
\begin{definition}
Let $(R,I)$ and $\tau_l = \{(R,I), \phi,U_l = ((-\infty, l), u_x); u_x \subset I, x \in (-\infty, l), l \in R\}$, then $\tau$ is called the left ray fuzzy topology on $(R,I)$.
\end{definition}
\begin{example}
Consider $U_{l1} , U_{l2}, U_{l3}$ are in a left ray fuzzy topology on $(R,I)$ such that $S(U_{l1}) = (-\infty, a), S(U_{l2}) =(-\infty, b), S(U_{l3}) = (-\infty, c )$, where $a > b > c$ with fuzzy membership values are given in the Graph 2.
\end{example}
\includegraphics[height=5cm,width=5 in]{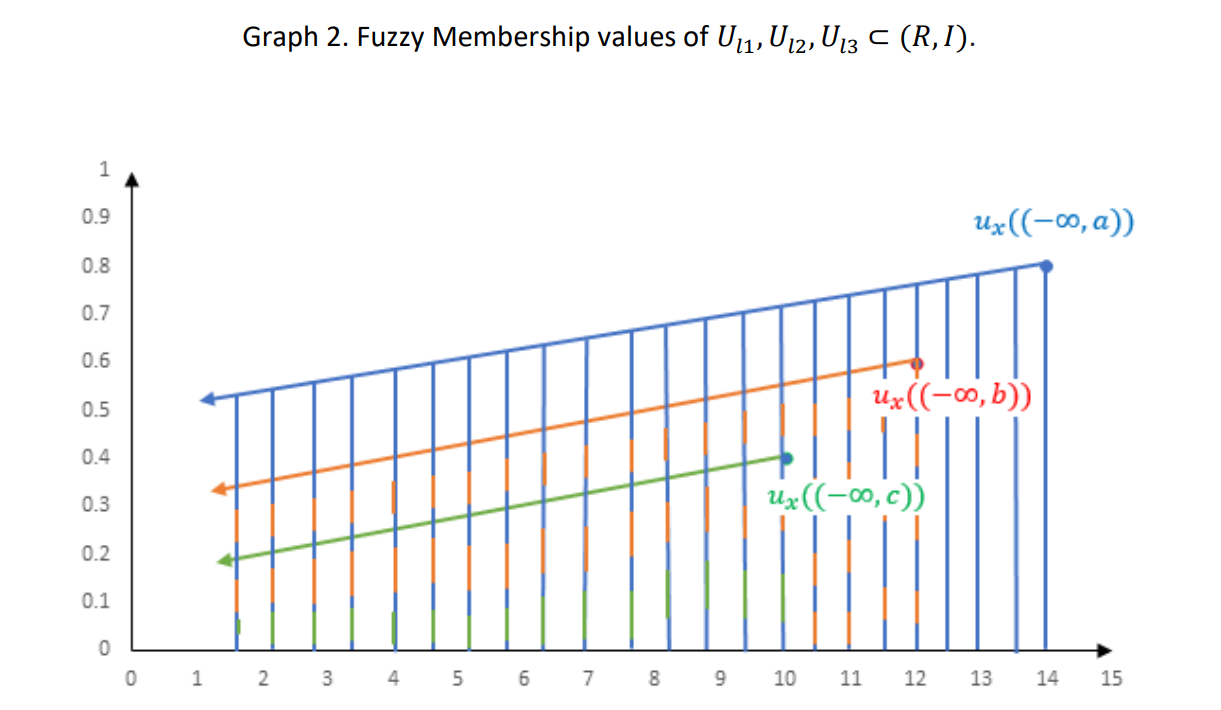}
\begin{definition}
	Let $(R,I)$ and $\tau = \{\emptyset, (R,I),U \subset (R,I)\}$; for each fuzzy point $P \in U$, there exists open interval $(a, b)$ such that $\{((a, b), v_x): (a, b) \subset U_0, v_x \subset u_x$ for all $a<x<b$ and $P_X \subset v_x$ for all $x \in P_0\}$, then $\tau$ is called the standard ( or usual ) fuzzy topology on $(R,I)$. 
\end{definition}
\begin{example}
	Let $S(U_1) = (a, b), S(U_2) = (h, q), S(U_3) = (c, d)$ be fuzzy subspace of $(R,I)$, where
	\[\mu_{U_I}=
	\begin{cases}
u_x, \;\; x \in U_i, \; i=1,2,3. \\
0, \;\;\; o.w
	\end{cases}
	\]
By using triangular fuzzy number, we may find the value of $\mu(x)$. See the Graph 3 below:
	\includegraphics[height=5cm,width=5 in]{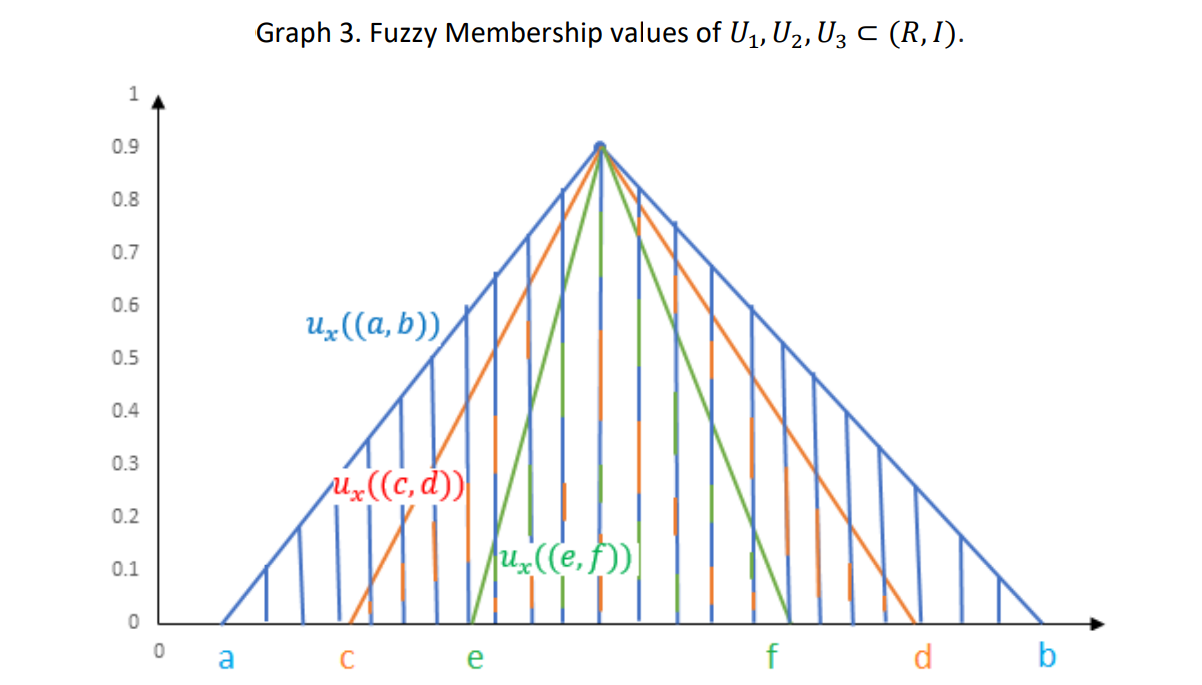}
\end{example}
\begin{definition}
	Let $((X,I), \tau)$ be a fuzzy topological space. A fuzzy subspace $D \subseteq (X,I)$ is said to be fuzzy dense if $Cl(D) = (X,I)$.
\end{definition}
\begin{definition}
	 Let $U$ be a fuzzy subspace of $(X,I)$, then a fuzzy point $P \in U$ is an isolated of $U$ iff there exists neighborhood $G$ containing $P$ such that $G \cap U = P$.
\end{definition}
\begin{Note}
	If the fuzzy points of $U$ that are not cluster fuzzy points of $U$ then are called isolated fuzzy points of $U$.
\end{Note}
\begin{definition}
	Let $U$ be a fuzzy subspace of $(X,I)$, a fuzzy point $P \in U$ is an exterior fuzzy point of $U$, iff there exists neighborhood $G$ containing $P$ such that $P \in G \subseteq (X,I)- U,$ and denoted by $Ext(U)$.
\end{definition}
\begin{definition}
 Let $U$ be a fuzzy subspace of $(X,I)$, a fuzzy point $P \in U$ is boundary fuzzy point of $U$, iff every neighborhood $G$ containing $P$ contains at least one fuzzy point of $U$ and at least one fuzzy point of $(X,I)-U$, and denoted by $Bd(U)$.
\end{definition}
\begin{theorem}
	 If $U$ is fuzzy subspace of the fuzzy topology $((X,I), \tau)$, then
	 \begin{enumerate}
	 	\item The fuzzy subspace $U$ of fuzzy topological space $(X,I)$ is fuzzy open iff $U = U^\circ	$
	 	\item $(A^\circ)^\circ = A^\circ$.
	 	\item $(A \cup B)^\circ	\supset A^\circ	 \cup B^\circ$ .
	 \end{enumerate}
\end{theorem}
\begin{proof}
	Prove (1) and (2) are omitted. The proof of (3) is illustrated below
	
	3. We have $\mu_{A^\circ} \subset  \mu_A \subset  \;\; max\{\mu_A,\mu_B\}, and \mu_B^\circ \subset \mu_B \subset \;\; max\{\mu_A, \mu_B\},$ $\max \{\mu_A^\circ, \mu_B^\circ\} \subset  \; \; max\{\mu_A, \mu_B\}$. Now $max\{\mu_A^\circ, \mu_B^\circ\}$ is a union of fuzzy neighborhood, hence is fuzzy neighborhood, so we get $max\{\mu_{A^\circ}, \mu_{B^\circ}\} \subset(max\{\mu_A, \mu_B\})^\circ$.
\end{proof}
\begin{theorem}
	If $U$ is fuzzy subspace of the fuzzy topology $((X,I), \tau)$, then
	\begin{enumerate}
		\item $Ext(A) = (X,I)\backslash	Cl(A)$.
		\item $Ext(\phi) = Int(X)$ and $Ext(\phi) = Int(X)$.
		\item Ext($U$) is the largest fuzzy neighborhood of $U^\circ$.
		\item Ext$(A \cup B) = Ext(A) \cap Ext(B)$.
	\end{enumerate}
\end{theorem}
\begin{proof}
	 Part (1), (2), and (3) trivial. Part (4) is proved as bellow:
	 
	 4. By definition Ext $(A \cup B)$ is equal to
$X \backslash Cl(max\{\mu_A, \mu_B\}) = X \backslash max\{\mu_{X \backslash Cl(A)}, \mu_{X \backslash Cl(B)}\} = min\{\mu_{X \backslash Cl(A)},  \mu_{X \backslash Cl(A)}\}$.
\end{proof}
\begin{theorem}
	If $U$ is fuzzy subspace of the fuzzy topology $((X,I), \tau)$, then
	\begin{enumerate}
		\item 	$Bd(C) = Cl(C) \backslash C^\circ$.
		\item The boundary of a fuzzy space is the boundary of the complement of the fuzzy space $Bd(U) = Bd(U^\circ)$.
		\item $Bd(A \cup  B) \subseteq Bd(A) \cup  Bd(B)$.
	\end{enumerate}
\end{theorem}
\begin{proof}
	Part (1) and (2) trivial. We proved part (3) as bellow:
	
	3. $max\{\mu_{Bd(A)}, \mu_{Bd(B)}\} = max\{\mu_{Cl(A \backslash A^\circ)}, \mu_{Cl(A \backslash A^\circ)}\}	 \supseteq Cl(\{\mu_A, \mu_B\} ) \backslash(\{\mu_A, \mu _B\} )^\circ= Bd \; max\{\mu_A, \mu_B\}$.
\end{proof}
\section{Separation Axioms}
In this section, separation axioms on fuzzy spaces are presented and studied by using some illustration 
examples. Later some theorems are proved to show the properties of fuzzy spaces in these notions.
\begin{definition}
	 Let $((X,I), \tau)$ be a fuzzy topological space. Then
	 \begin{enumerate}
	 	\item [(a)] (fuzzy $T_0-$Space) A fuzzy space $(X,I)$ is a fuzzy $T_0$-space if for each pair distinct fuzzy points $p, q \in (X,I)$ there is a neighborhood $U \subset (X,I)$ such that $U$ contains one of $p$ or $q$ but not the other.
	 	\item [(b)] (fuzzy $T_1$-Space) A fuzzy space $(X,I)$ is a fuzzy $T_1$-space if for each pair of distinct fuzzy points $p, q \in	(X,I)$ there are two neighborhoods $G, H \subset (X,I)$ such that $p \in G$ but $q \not \in  G$ or $q \in H$ but $p \not \in H$.
	 	\item [(c)] (fuzzy $T_2$-Space) A fuzzy space $(X,I)$ is a fuzzy $T_2$-space if for each pair of distinct fuzzy points $p, q \in X$ there are two disjoint neighborhoods $G, H \subset X$ such that $p \in G, q \in H$.
	 	\item  [(d)] (fuzzy Regular Space) A fuzzy space $(X,I)$ is fuzzy regular if for each fuzzy point $p \in (X,I)$ and each closed fuzzy  subspace $F \subset (X,I)$ such that $p \not \in  F$ there are two disjoint neighborhoods $G, H \subset (X,I)$ such that $p \in G, F \subset H$.
	 	\begin{Note}
	 		A fuzzy regular fuzzy $T_1$-space is called a fuzzy $T_3$-space.
	 	\end{Note}
	 	\item [(e)] (fuzzy Normal Space) A fuzzy space $(X,I)$ is fuzzy normal if for each pair $F_1
	 	, F_2$ of disjoint closed fuzzy 	subspaces of $(X,I)$, there are two disjoint neighborhoods $G, H$, so that $F_1 \subset G, F_2 \subset H.$	
	 	 \begin{Note}
A fuzzy normal fuzzy $T_1$-space is called a fuzzy $T_4$-space.
	 	 \end{Note} 
 		 \end{enumerate}
 	 \begin{example}
 	 	Fuzzy $T_0$-space is not fuzzy $T_1$-space.
 	 	
 	 	Let $X= \{a,b\}$ and $p=(a,\{0,\frac{1}{2}\}) , q=(b,\{0,\frac{1}{3}\}), s=(\{a,b\},\{0,\frac{1}{6}\})$ are the fuzzy points of the fuzzy space 
 	 	$(X,I)$ with the fuzzy topology
 	 	$\tau = \{\emptyset,(X,I),u_1=\{(a,\{0,\frac{1}{2}\})\},$ $u_2  =\{(\{a,b\},\{0,\frac{1}{6}\})\},u_3=\{(\{0,\frac{1}{2}\}),(\{a,b\},\{0,\frac{1}{6}\})\}\}.$

 	 	Notice that $p \not = q$ with $p \in u_1$ and $q \not \in u_1$. For $p \not \in s$, we have $p \in u_1$ and $s \not \in  u_1$. Also, for $q \not = s$, we see that $s \in u_2$ and $q \not \in u_2$. So, $(X,I)$ is a fuzzy $T_0$- space. To show $(X,I)$ is not fuzzy $T_1$-space. For $p \not =q$, there is no neighborhood contains $q$, so it is not fuzzy $T_1$-space.

 	 \end{example}
\end{definition}
\section{Conclusion} \label{FP-Arch conculusion}
\hspace{\parindent}
The new approach given by Dib in 1994 is reviewed in this paper. We fill the lack of Dib's article by introducing 
isolated point, exterior, and other notions in fuzzy space. Separation axioms are illustrated, and some of its 
properties are proved under Dip's approach. Some examples are illustrated graphically to be deal with fuzzy space.
\bibliographystyle{unsrt}  
\bibliographystyle{unsrt}  


\end{document}